\documentclass{article}
\usepackage{amsmath,amsfonts,amssymb,amscd,amsthm,amsbsy}
\usepackage{graphicx}

\author{Luis  Caffarelli and Luis Silvestre}

\title{The Evans-Krylov theorem for non local fully non linear equations}

\newlength{\hchng}
\newlength{\vchng}
\newtheorem{thm}{Theorem}[section]
\newtheorem{prop}[thm]{Proposition}
\newtheorem{cor}[thm]{Corollary}
\newtheorem{lemma}[thm]{Lemma}

\newtheorem{preremark}[thm]{Remark}
\newenvironment{remark}{\begin{preremark}\rm}{\medskip \end{preremark}}
\numberwithin{equation}{section}

\newcommand{\norm}[1]{\left\Vert#1\right\Vert}
\newcommand{\abs}[1]{\left\vert#1\right\vert}
\newcommand{\set}[1]{\left\{#1\right\}}
\newcommand{\R}{\mathbb R}
\newcommand{\eps}{\varepsilon}

\newcommand{\grad} {\nabla}
\newcommand{\lap} {\triangle}
\newcommand{\bdary} {\partial}
\newcommand{\dx} {\; \mathrm{d} x}
\newcommand{\dd} {\; \mathrm{d}}

\DeclareMathOperator{\supp}{supp}

\newcommand{\LI}{\mathcal{L}}

\newcommand{\si}{\delta}
\newcommand{\Mp} {\mathrm{M}^+}
\newcommand{\Mm} {\mathrm{M}^-}

\newcommand{\ind}{a}

\begin{document}
\maketitle

\begin{abstract}
We prove a regularity result for solutions of a purely integro-differential Bellman equation. This regularity is enough for the solutions to be understood in the classical sense. If we let the order of the equation approach two, we recover the theorem of Evans and Krylov about the regularity of solutions to concave uniformly elliptic partial differential equations.
\end{abstract}

\section{Introduction}
\label{s:main}
In 1982, L. Evans and N. Krylov proved independently (\cite{MR649348} and \cite{MR661144}) the following celebrated interior regularity result for elliptic partial differential equations: If $u$ is a bounded solution to $F(D^2 u)$ in $B_1$, where $F$ is uniformly elliptic and concave, then $u \in C^{2,\alpha}(B_{1/2})$ for some $\alpha > 0$. In this paper we prove a nonlocal version of that theorem. We prove that solutions to concave integro-differential equations of order $\sigma$ have regularity $C^{\sigma+\alpha}$ for some $\alpha>0$. This is enough regularity to consider the solutions to be classical.

The equations we study arise in stochastic control problems with jump processes (see for example \cite{MR1355243}, \cite{MR1424785}). In \cite{MR1355243} a $C^{2,\alpha}$ regularity of the solutions of Bellman equations for Levy processes is obtained, but the equation is required to have a uniformly elliptic second order part which is ultimately the source of the regularity. In \cite{abels2007aap} a purely integro-differential Bellman equation is studied. They only consider the case of the maximum of two linear operators. They obtain solutions in the fractional Sobolev space $H^{\sigma/2}$ up to the boundary and $H^\sigma$ in the interior of the domain. As they point out in their paper, the solutions to these equations are expected to be more regular in the interior of the domain.

In this paper we consider purely integro-differential equations and obtain an interior regularity result. Since we do not require our equations to have a second order part, our estimate comes only from the regularization effects of the integrals.

The constants in our estimates do not blow up as $\sigma \to 2$, so we can recover the usual Evans-Krylov theorem as a limit case. It is interesting to follow what the ideas of the proofs become as $\sigma \to 2$. Interestingly, the ideas we present in this paper provide a different proof of the Evans-Krylov theorem for second order elliptic equations.

We consider the equation
\begin{equation}
\label{e:main}
Iu(x) := \inf_{\ind \in \mathcal{A}} L_\ind u(x) = \inf_{\ind \in \mathcal{A}} \int_{\R^n} (u(x+y) + u(x-y) - 2u(x)) K_\ind (y) \dd y = 0
\end{equation}
As in \cite{CS}, we will choose each linear operator $L_\ind$ in some class $\LI$. Consequently, the operator $I$ will be elliptic with respect to $\LI$ in the sense described in \cite{CS2}.

We describe below the appropriate classes of linear operators that we will use in this paper.

We say that an operator $L$ belongs to $\LI_0$ if its corresponding kernel $K$ satisfies the uniform ellipticity assumption.
\begin{equation} \label{e:ellipticity}
(2-\sigma) \frac \lambda {|y|^{n+\sigma}} \leq K(y) \leq (2-\sigma) \frac \Lambda {|y|^{n+\sigma}}.
\end{equation}

The ellipticity assumption \eqref{e:ellipticity} is the essential assumption that leads to a local regularization. Our proofs, as usual, involve an improvement of oscillation of the solution to the equation (or an operator applied to it) in a decreasing sequence of balls around a point in the domain. Since the equations are nonlocal, every argument in our proofs will have to take into account the influence of the values of the solution at points outside those balls. We will often need to say that the part of the integral in \eqref{e:main} outside a neighborhood of the origin is a smooth enough function. That is why we define the following classes of smooth kernels.

We say that $L \in \LI_1$ if, in addition to \eqref{e:ellipticity}, the kernel $K$ is $C^1$ away from the origin and satisfies
\begin{equation}
\grad K_{\ind}(y) \leq \frac C {|y|^{n+1+\sigma}}.
\end{equation}

Finally, we say $L \in \LI_2$ if the kernel is $C^2$ away from the origin and satisfies
\begin{equation} \label{e:kernelc2}
D^2 K_{\ind}(y) \leq \frac C {|y|^{n+2+\sigma}}
\end{equation}

We consider the corresponding maximal operators
\begin{align*}
\Mp_0 u(x)&= \sup_{L \in \LI_0} Lu(x) = \int_{\R^n} \frac{\Lambda (\si u(x,y))^+ - \lambda (\si u(x,y))^-}{|y|^{n+\sigma}} \dd y\\
\Mp_1 u(x)&= \sup_{L \in \LI_1} Lu(x) \\
\Mp_2 u(x)&= \sup_{L \in \LI_2} Lu(x) \\
\end{align*}

Recall that we write $\si u(x,y) = (u(x+y)+u(x-y) - 2u(x))$ as in \cite{CS}. Naturally we have the inequalities $\Mp_0 u \geq \Mp_1 u \geq \Mp_2 u$. The minimal operators $\Mm$ are defined likewise.

We do not know a closed form for $\Mp_1$ or $\Mp2$. Since $\LI_2 \subset \LI_1 \subset \LI_0$, we have the relations $\Mp_2 u \leq \Mp_1 u \leq \Mp_0 u$ and $\Mp_2 u \geq \Mp_1 u \geq \Mp_0 u$.

Our main result states that under the hypothesis that all operators $L_\ind$ belong to $\LI_2$, the solutions are classical in the sense that there is enough regularity so that all integrals are well defined and H\"older continuous.

\begin{thm}
\label{t:main}
Assume every $L_\ind$ in \eqref{e:main} belongs to the class $\LI_2$. If $u$ is a bounded function in $\R^n$ such that $I u=0$ in $B_1$, then $u \in C^{\sigma+\alpha}(B_{1/2})$. Moreover
\begin{equation}
\norm{u}_{C^{\sigma+\alpha}(B_{1/2})} \leq C \norm{u}_{L^\infty(\R^n)}
\end{equation}
\end{thm}

The interested reader may verify, in following the arguments, that global boundedness of $u$ may be substituted by an appropriate moderated growth at infinity (see remark at the end of this paper).

For values of $\sigma$ less or equal to $1$, this theorem does not provide any improvement with respect to the $C^{1,\alpha}$ estimates in \cite{CS2}. Thus, for the purpose of proving Theorem \ref{t:main} we will assume $\sigma$ to be strictly larger than $1$ in this paper. The result becomes most interesting when $\sigma$ is close to $2$ and $\sigma+\alpha > 2$.

Note that the result of the theorem remains true if $I$ is convex instead of concave (a $\sup$ of linear operators instead of an $\inf$). Indeed, we can transform one situation in the other by considering the equation $-I(-u)=0$.

In previous papers (\cite{CS} and \cite{CS2}) we started developing the regularity theory for nonlocal equations. In \cite{CS} we obtained a nonlocal version of Krylov-Safonov theory with estimates that do not blow up as $\sigma \to 2$. This allowed us to obtain $C^{1,\alpha}$ estimates for general fully nonlinear integro-differential equations that are translation invariant. In \cite{CS2}, we extended those results to variable coefficient equations using perturbative methods. In this paper, we use the results in our previous two papers extensively.

\section{A regularization procedure}
\label{s:regularizationProcedure}

In this section we show a simple technique to approximate uniformly the solutions to the integro-differential equation \eqref{e:main} by $C^{2,\alpha}$ functions that solve an approximate equation with the same structure. This procedure works exclusively for integro-differential equations and cannot be done using only second order equations. It makes it unnecessary to use sup- or inf- convolutions and simplifies the technicalities of several proofs. Essentially the idea is that if we prove an estimate assuming the solutions $u$ is $C^{2,\alpha}$ (but the estimate does not depend on the $C^{2,\alpha}$ norm), then we can pass to the limit using this approximation technique to extend the estimate to all viscosity solutions. In this respect, the technicalities in the integro-differential setting simplify very much compared to the second order counterpart.

\begin{lemma} \label{l:approximationByC2}
Let $u$ be a continuous function in $\R^n$ solving \eqref{e:main} with every $L_\ind$ belongs to the class $\LI_2$ (resp. $\LI_1$ or $\LI_0$). There is a sequence of \emph{regularized} equations in the same class
\begin{align*}
 I^\eps u^\eps &= \inf_\ind L^\eps_\ind u^\eps = 0 && \text{in } B_1 \\
 u^\eps &= u && \text{in } \R^n \setminus B_1
\end{align*} 
so that the solutions $u^\eps$ are $C^{2,\alpha}$ in the interior of $B_1$ for every $\eps>0$ and $\lim_{\eps \to 0} u^\eps = u$ uniformly in $B_1$.
\end{lemma}

\begin{proof}
Let $\eta$ be a smooth function such that
\begin{align*}
0\leq \eta &\leq 1  \ \text{in } \R^n, \\
\eta &= 0  \ \text{in } \R^n \setminus B_1, \\
\eta &= 1  \ \text{in } B_1,
\end{align*}
and let $\eta_\eps(x) = \eta(x/\eps)$.

Let us consider the following regularized kernels 
\[ K_\ind^\eps (y) = \eta_\eps(x) \lambda \frac{2-\sigma}{|y|^{n+\sigma}} + (1-\eta_\eps(x)) K_\ind(y) \]
 
Correspondingly, we define
\begin{align*}
L^\eps_\ind v &= \int_{\R^n} \si v(x,y) K_\ind^\eps(y) \dd y \\
I^\eps v &= \inf_\ind L_\ind^\eps v
\end{align*}

Note that if $L_\ind \in \LI_i$ then also $L^\eps_\ind \in \LI_i$ for $i=0,1,2$.

Let $u^\eps$ be the solution of the following Dirichlet problem:
\begin{align*}
 I^\eps u^\eps &= \inf_\ind L^\eps_\ind u^\eps = 0 && \text{in } B_1 \\
 u^\eps &= u && \text{in } \R^n \setminus B_1
\end{align*} 

The solution $u^\eps$ to this problem is $C^{2,\alpha}$ by theorem 6.6 in \cite{CS2}.

It is clear that if $v \in C^2(x)$ and $|v(y) - v(x) - (y-x) \cdot \grad v(x)| \leq M |y-x|^2$ in $B_1$, then
\[ |I^\eps v (x) - I v(x)| \leq CM\eps^{2-\sigma} \]
so $\norm{I^\eps-I} \to 0$ as $\eps \to 0$ (recall $\sigma<2$), where the norm $\norm{I^\eps-I}$ is computed in the sense of definition 2.2 in \cite{CS2}. Then, by lemma 4.9 in \cite{CS2}, $u^\eps$ converges to $u$ uniformly in $\overline B_1$ as $\eps \to 0$.
\end{proof}

\begin{remark}
The concavity of $I$ is not used in Lemma \ref{l:approximationByC2}. The exact same idea works for equations of the type
\[ Iu(x) := \sup_b \inf_\ind L_{\ind b} u(x) = \sup_b \inf_\ind \int_{\R^n} (u(x+y) + u(x-y) - 2u(x)) K_{\ind b} (y) \dd y = 0 \]
\end{remark}

%\begin{remark}
%We could find a similar result for second-order elliptic PDEs by approximating them with integro-differential equations first, and then doing an approximation procedure as above.
%\end{remark}

\section{Average of subsolutions is a subsolution}
\label{s:average}
The main ingredient in the Evans-Krylov theorem is the fact that concavity of the equation makes second order incremental quotients subsolutions of the linearized equation. In order to prove that, one first observes that an average of solutions to a concave equation is a subsolution to the same equation.

In this section we prove that also in the non local case the average of subsolutions to a concave equation is a subsolution of the same equation. This is obvious from the equation if the solutions are classical. For viscosity solutions we can prove it quickly using the approximation technique of section \ref{s:regularizationProcedure}.

\begin{prop} \label{p:averageIsSubsolution}
Let $u$ and $v$ be subsolutions of $Iu=0$ and $Iv=0$ in a domain $\Omega$, $u,v$ continuous in $\R^n$, then $I(u+v)/2 \geq 0$ in $\Omega$.
\end{prop}

\begin{proof}
The proposition is obvious if $u,v \in C^2$ by the concavity of $I$. So we used the regularization procedure described in section \ref{s:regularizationProcedure}.

Let $I^\eps u^\eps = 0$ and $I^\eps v^\eps = 0$ be the approximate equations of Lemma \ref{l:approximationByC2}. The functions $u^\eps$ and $v^\eps$ are $C^2$ so $I^\eps (u^\eps+v^\eps)/2 \geq 0$ in $\Omega$. Since $u^\eps \to u$ and $v^\eps \to v$ uniformly in $\Omega$ and $I^\eps \to I$, then $I(u+v)/2 \geq 0$ in $\Omega$ by Lemma 4.9 in \cite{CS2}.
\end{proof}

The same idea shows that any average of solutions is a subsolution. In particular we have the following proposition.

\begin{prop} \label{p:mollified}
Let $u$ be a solution of $Iu = 0$ in $B_1$ and $\eta$ be a mollifier, i.e.
\begin{enumerate}
\item $\eta \geq 0$
\item $\int \eta = 1$.
\item $\supp \eta \subset B_\delta$.
\end{enumerate}
we have $I(\eta \ast u) \geq 0$ in $B_{1-\delta}$.
\end{prop}

\section{The linear theory of integro-differential equations}
\label{s:linearTheory}

In this section we present some regularity theorems for linear integro-differential equations with constant coefficients. Naturally in this simple case, we can easily obtain more powerful results than for the nonlinear case. The results we present in this section are just the ones that we will need in the rest of this paper.

\begin{thm} \label{t:l1}
Let $L$ be an integro-differential operator in the class $\LI_1$ with $\sigma \geq \sigma_0 >1$. Suppose that $u$ is an integrable function in the weighted space $L^1(\R^n,\frac{1}{1+|y|^\sigma})$ that solves the equation $L u = 0$ in $B_1$, then $u \in C^{2,\alpha}(B_{1/2})$ and we have the estimates
\[ \norm{u}_{C^{2,\alpha}(B_{1/2})} \leq C_k \norm{u}_{L^1(\R^n,\frac{1}{1+|y|^\sigma})} \]
The value of the constant $C$ and $\alpha$ depends on $n$, $\lambda$, $\Lambda$ and $\sigma_0$ but not on $\sigma$.
\end{thm}

\begin{proof}
We will prove the apriori estimate. The regularity estimate for a weak or viscosity solution follows by mollifying the solution or using the regularization procedure of the previous section.

First we apply theorem 2.8 in \cite{CS2} to obtain that $u \in C^{1,\alpha}(B_{3/4})$ and obtain the estimate
\[ ||u||_{C^1(B_{3/4})} \leq C \norm{u}_{L^1(\R^n,\frac{1}{1+|y|^\sigma})}. \]

The idea is to apply the same $C^{1,\alpha}$ estimate to every directional derivative $u_e$. Since we do not have an $L^\infty$ estimate of $u_e$ outside of $B_{3/4}$ we have to use our usual \emph{integration by parts} trick. We know that
\[ \int_{\R^n} u_e(y) K(x+y) \dd y = 0\]

Let $\eta$ be a smooth cutoff function such that
\begin{align*}
0 &\leq \eta \leq 1  && \text{in } \R^n, \\
\eta &= 0 && \text{outside } B_{3/4}, \\
\eta &= 1 && \text{in } B_{5/8}.
\end{align*}
 
We compute 
\begin{align*}
 \abs{\int_{\R^n} u_e(y) \eta(y) K(x+y) \dd y} &=  \abs{\int_{\R^n} u_e(y) (\eta(y)-1) K(x+y) \dd y} \\
 &=  \abs{\int_{\R^n} u_e(y) (\eta_e(y) K(x+y) + (\eta(y)-1) K_e(x+y)) \dd y} \\
 &\leq C \norm{u}_{L^1(\R^n,\frac{1}{1+|y|^\sigma})}.
\end{align*}

Thus we can apply Theorem 6.1 in \cite{CS2} to conclude that $u_e \in C^{1,\alpha}$ for every direction $e$ and thus $u \in C^{2,\alpha}$ (Note that we are using $\sigma>1$ here).
\end{proof}

In order to have better interior regularity estimates than $C^{2,\alpha}$, we would need to impose more regularity to the kernel $K$ in $L$ than $C^1$ away from the origin.

Next theorem says that in $L^2$ all linear operators have a comparable norm.

\begin{thm} \label{t:l2}
Let $L_0$ and $L_1$ be two linear integro-differential operators in the class $\LI_0$. Suppose that $L_0 u \in L^2(\R^n)$ then $L_1 u \in L^2(\R^n)$.
\end{thm}

\begin{proof}
Since we are dealing with $L^2$ norms, and translation invariant linear operators, we will use the Fourier transform to prove this theorem.

Given a function $u$ and $y \in \R^n$, we have $\widehat{\si u(x,y)} = (u(.+y)+u(.-y)-2u)\widehat{\phantom{.}} = (e^{i y \cdot \xi} + e^{-i y \cdot \xi} - 2) \hat u(\xi) = 2(\cos(y \cdot \xi)-1) \hat u(\xi)$. We use this identity to compute the symbol $s(\xi)$ of an operator $-L$ as a pseudo differential operator.

\begin{align*}
-\widehat{Lu}(\xi) &= \left( \int_{\R^n} \si u(x,y) K(y) \dd y \right)\widehat{\phantom{.}} (\xi) \\
&= \left( \int_{\R^n} 2(1-\cos(y \cdot \xi)) K(y) \dd y  \right)\hat u (\xi) =: s(\xi) \hat u(\xi)
\end{align*}

Note that for every $\xi$ function $(1-\cos(y \cdot \xi))$ is $C^2$ and bounded, so the integral in the right hand side is well defined. Let us estimate it from above and below.

For any $R>0$,
\begin{align*}
s(\xi) = \int_{\R^n} 2(1-\cos(y \cdot \xi)) K(y) \dd y &\leq \int_{B_R} 2 |y \cdot \xi|^2 (2-\sigma) \frac{\Lambda}{|y|^{n+\sigma}} \dd y + \int_{\R^n \setminus B_R} 2 (2-\sigma) \frac{\Lambda}{|y|^{n+\sigma}} \dd y \\
&\leq C |\xi|^2 R^{2-\sigma} + C \frac{(2-\sigma)}{\sigma} R^{-\sigma}
\end{align*}
so we obtain $s(\xi) \leq C |\xi|^\sigma$ by choosing $R = |\xi|^{-1}$.

On the other hand note that $(1-\cos(y \cdot \xi))$ is nonnegative and so is $K(y)$. So the integrand is nonnegative and we have
\begin{align*}
s(\xi) = \int_{\R^n} 2(1-\cos(y \cdot \xi)) K(y) \dd y &\geq \int_{B_{|\xi|^{-1}/2}} \frac{1}{4}|y \cdot \xi|^2 (2-\sigma) \frac{\lambda}{|y|^{n+\sigma}} \dd y \\
&\geq c |\xi|^{-\sigma}
\end{align*}
So the symbol $s(\xi)$ is comparable to $|\xi|^{-\sigma}$ for any operator $L$ in $\LI_0$. Thus by classical Fourier analysis we have that $L_1 L_0^{-1}$ has a bounded symbol and maps $L^2$ functions into $L^2$.
\end{proof}

The following theorem is a direct combination of Theorems \ref{t:l1} and \ref{t:l2}.

\begin{thm} \label{t:linearL2estimate-local}
Let $L$ be an integro-differential operator in the class $\LI_1$ with $\sigma \geq \sigma_0 >1$. Suppose that $u$ is a function in the weighted space $L^1(\R^n,\frac{1}{1+|y|^\sigma})$ that solves the equation $L u = f$ in $B_1$ for some $f \in L^2$. Let $L_1$ be an operator in $\LI_0$, then $L_1 u \in L^2(B_{1/2})$ and 
\[\norm{u}_{L^2(B_{1/2})} \leq C \left( \norm{u}_{L^1(\R^n,\frac{1}{1+|y|^\sigma}) + \norm{f}_{L^2(B_1)}} \right) \]
for some constant $C$ depending on $n$, $\lambda$, $\Lambda$ and $\sigma_0$.
\end{thm}

\begin{proof} 
Consider the function $v$ that solves 
\[ L v = f \chi_{B_1} \text{   in } \R^n. \]

From Theorem \ref{t:l2} we get that $L_1 v \in L^2(\R^n)$. Theorem \ref{t:l2} can also be applied to $L_2 = (-\lap)^{\sigma/2}$, thus $v$ is in the homogeneous fractional Sobolev space $\dot H^{\sigma/2}$. By Sobolev embedding $v \in L^p(\R^n)$ where $p = 4n/(2n-\sigma)$. In particular $v \in L^1(\R^n,\frac{1}{1+|y|^\sigma})$

Now we apply theorem \ref{t:l1} to $u-v \in L^1(\R^n,\frac{1}{1+|y|^\sigma})$ to conclude the proof.
\end{proof}

\begin{remark}
The ellipticity constants $\lambda$ and $\Lambda$ of $L_0$ and $L_1$ do not need to coincide. Indeed if each $L_i$ is elliptic with constants $\lambda_i$ and $\Lambda_i$, then they would both be elliptic with respect to the constants $\min(\lambda_0,\lambda_1)$ and $\max(\Lambda_0,\Lambda_1)$.
\end{remark}

\section{Subsolutions in $L^1$ are bounded above}

The theorem below is a weak version of the mean value theorem. Its proof uses the same ideas as the proof of the Harnack inequality in \cite{CS}. We include it here for completeness.

\begin{thm} \label{t:weakharnack2}
Let $u$ be a function such that $u$ is continuous in $\overline B_1$, assume that
\begin{align*}
\int_{\R^n} \frac{|u(y)|}{1+|y|^{n+\sigma}} \dd y &\leq C_0 \\
\Mp u &\geq -C_0 \ \ \text{in }B_1
\end{align*}
then
\[ u(x) \leq C C_0 \text{ in } B_{1/2} \]
for every $x \in B_{1/2}$, where $C$ is a {\em universal constant}.
\end{thm}

\begin{proof}
Dividing $u$ by $C_0$, we can assume without loss of generality that $C_0=1$.

Let us consider the minimum value of $t$ such that
\[ u(x) \leq h_t(x) := t (1-|x|)^{-n} \text { for every } x \in B_1. \]
There must be an $x_0 \in B_1$ such that $u(x_0) = h_t(x_0)$, otherwise we could make $t$ smaller. Let $d = (1-|x_0|)$ be the distance from $x_0$ to $\bdary B_1$.

For $r=d/2$, we want to estimate the portion of the ball $B_r(x_0)$ covered by $\{u < u(x_0)/2\}$ and by $\{u > u(x_0)/2\}$. We will show that $t$ cannot be too large. In this way we obtain the result of the theorem, since the upper bound $t<C$ implies that $u(x) < C (1-|x|)^{-n}$.

Let us first consider $A := \{u > u(x_0)/2\}$. By assumption, we have $u \in L^1(B_1)$, thus
\begin{align*}
|A \cap B_1| &\leq C \abs{\frac{2}{u(x_0)}} \\ 
&\leq C t^{-1} d^n
\end{align*}

Whereas $|B_r| = C d^n$, so if $t$ is large, $A$ can cover only a small portion of $B_r(x_0)$ at most.
\begin{equation} \label{e:a0}
\abs{ \{u>u(x_0)/2\} \cap B_r(x_0) } \leq C t^{-1} \abs{B_r}
\end{equation}

In order to get a contradiction, we will show that $\abs{\{u<u(x_0)/2\} \cap B_r(x_0)} \leq (1-\alpha) B_r$ for a positive constant $\alpha$ independent of $t$. 

We estimate $\abs{\{u<u(x_0)/2\} \cap B_{\theta r} (x_0)}$ for $\theta>0$ small. For every $x \in B_{\theta r} (x_0)$ we have $u(x) \leq h_t(x) \leq (d-\theta d/2)^{-n} \leq u(x_0) (1-\theta/2)^{-n}$, with $(1-\theta/2)^{-n}$ close to one.

Let us consider 
\[ v(x) = (1-\theta/2)^{-n} u(x_0) - u(x) \]
so that $v \geq 0$ in $B_{\theta r}$, and also $\Mm v \leq 1$ since $\Mp u \geq -1$. We would want to apply Theorem 10.4 in \cite{CS} (the $L^\eps$ estimate) to $v$. The only problem is that $v$ is not positive in the whole domain but only on $B_{\theta r}$. In order to apply such theorem we have to consider $w=v^+$ instead, and estimate the change in the right hand side due to the truncation error.

We want to find an upper bound for $\Mm w = \Mm v^+$ instead of $\Mm v$. We know that
\[
\Mm v(x) = (2-\sigma) \int_{\R^n} \frac{\lambda \si v(x,y)^+ - \Lambda \si v(x,y)^-}{|y|^{n+\sigma}} \dx 
\leq 1.
\]

Therefore, if $x \in B_{\theta r /2}(x_0)$,
\begin{align}
\Mm w &= (2-\sigma) \int_{\R^n} \frac{\lambda \si w(x,y)^+ - \Lambda \si w(x,y)^-}{|y|^{n+\sigma}} \dd y \\
& \leq 1 + (2-\sigma) \int_{\R^n \cap \{v(x+y)<0\}} -\Lambda \frac{v(x+y)}{|y|^{n+\sigma}} \dd y \\
& \leq 1 + (2-\sigma) \int_{\R^n \setminus B_{\theta r/2}} \Lambda \frac{(u(x+y)-(1-\theta/2)^{-n} u(x_0))^+}{|y|^{n+\sigma}} \dx \label{e:aa1} \\
& \leq 1 + C (2-\sigma) (\theta r)^{-n-\sigma} \int_{\R^n} \Lambda \frac{|u(y)|}{1+|y|^{n+\sigma}} \dx \leq C (\theta r)^{-n-\sigma}
\end{align}

Now we can apply Theorem 10.4 from \cite{CS} to $w$ in $B_{\theta r/2}(x_0)$. Recall $w(x_0) =((1-\theta/2)^{-n}-1)u(x_0)$, we have
\begin{align*}
\left\vert \set{ u < \frac{u(x_0)}{2} } \cap B_{\frac{\theta r}{4}(x_0)} \right\vert &= | \{ w > u(x_0) ((1-\theta/2)^{-n}-1/2) \} \cap B_{\theta r /4}(x_0) | \\
\leq C (\theta r)^n \big( ((1-&\theta/2)^{-n}-1 ) u(x_0)+C (\theta r)^{-n-\sigma} (r\theta)^\sigma \big)^\eps \left( u(x_0) ((1-\theta/2)^{-n}- \frac 1 2 ) \right)^{-\eps} \\
\leq C (\theta r)^n \big( ((1-&\theta/2)^{-n}-1)^\eps + \theta^{-n \eps} t^{-\eps} \big)
\end{align*}

Now let us choose $\theta>0$ so that the first term is small:
\[ C (\theta r)^n ((1-\theta/2)^{-n}-1)^\eps \leq \frac{1}{4} \abs{B_{\theta r/2}} \ . \]

Notice that the choice of $\theta$ is independent of $t$. For this fixed value of $\theta$ we observe that if $t$ is large enough, we will also have
\[ C (\theta r)^n \theta^{-n \eps} t^{-\eps} \leq \frac{1}{4} \abs{B_{\theta r/2}} \]
and therefore
\[ | \{ u < u(x_0)/2 \} \cap B_{\theta /r 4}(x_0) | \leq \frac{1}{2} \abs{B_{\theta r/4}(x_0)} \]
which implies that for $t$ large
\[ | \{ u > u(x_0)/2 \} \cap B_{\theta r /4}(x_0) | \geq c \abs{B_r} \ . \]
But this contradicts \eqref{e:a0}. Therefore $t$ cannot be large. Rescaling back, we obtain
\[ u(x) \leq C C_0, \]
for any $x$ in $B_{1/2}$.
\end{proof}

\section{Each $L_\ind u$ is bounded}
\label{s:eachLaBounded}
The idea of this section is to show that averages of second order incremental quotients are subsolutions (of the maximal operator $\Mp$). In particular, each $v_\ind := L_\ind u$ is a subsolution to
\begin{align*}
v_\ind &\geq 0 \in B_1 \\
\Mp v_\ind &\geq -C \in B_1
\end{align*}
Then we would estimate the integral of $v_\ind$ in $B_{1/2}$ and use Theorem \ref{t:weakharnack2} to prove that $L_\ind u$ is bounded.

\begin{lemma} \label{l:truncatedLisSubsolution}
Assume $u \in C^2$. Then for every kernel $K$ corresponding to an operator $L \in \LI_0$ and every bump function $b$ such that
\begin{align*}
0 &\leq b(x) \leq 1 \text{ in } \R^n, \\
b(x) &= b(-x) \text{ in } \R^n, \\
b(x) &= 0 \text{ in } \R^n  \setminus B_{1/2},
\end{align*} 
we have
\[ \Mp_2 \left( \int_{B_{\R^n}} \si u(x,y) K(y) b(y) \dd y \right) \geq 0 \qquad \text{in } B_{1/2}\]
\end{lemma}

\begin{proof}
Let $\phi_k$ be the $L^1$ function $\phi_k(y) =  \chi_{\R^n \setminus B_{1/k}}(y) K(y) b(y)$. Since $u \in C^2$, we can approximate the value of the integral uniformly by
\[ \int_{B_{\R^n}} \si u(x,y) K(y) b(y) \dd y = \lim_{k \to \infty} \int_{\R^n} \si u(x,y) \phi_k(y) \dd y = \lim_{k \to \infty} u \ast \phi_k - \norm{\phi_k}_{L^1} u. \]

Applying Proposition \ref{p:averageIsSubsolution}, we have
\[ I \left(u \ast  \frac{\phi_k}{\norm{\phi_k}_{L^1}} \right) \geq 0\]
On the other hand, we know that $I u = 0$, then by Lemma 5.8 in \cite{CS} and the fact that $\Mp$ is homogeneous,
\[ \Mp_2 (u \ast \phi_k - \norm{\phi_k}_{L^1} u) = \norm{\phi_k}_{L^1} \Mp_2 \left(u \ast  \frac{\phi_k}{\norm{\phi_k}_{L^1}} - u \right) \geq 0 \]

Since we have $\Mp_2(u \ast \phi_k - \norm{\phi_k}_{L^1} u) \geq 0$ for every $k>0$. Then the result follows by Lemma 4.3 in \cite{CS2} by taking limit as $k\to \infty$.
\end{proof}

\begin{lemma} \label{l:fullLisSubsolution}
Assume $u \in C^2$. Then there is a universal constant $C$ such that for every operator $L \in \LI_2$ \[ \Mp_2 \left( Lu \right) \geq -C \qquad \text{in } B_{1/2}\]
\end{lemma}

\begin{proof}
 As in the proof of Lemma \ref{l:truncatedLisSubsolution}, we let $\phi_k$ be the $L^1$ function $\phi_k(y) =  \chi_{\R^n \setminus B_{1/k}}(y) K(y)$ and we approximate the value of the integral uniformly by
\[ Lu(x) = \lim_{k \to \infty} \int_{\R^n} \si u(x,y) \phi_k(y) \dd y = \lim_{k \to \infty} (u \ast \phi_k - u \norm{\phi_k}_{L_1}) \]

Let $b$ be a bump function such that
\begin{align*}
0 &\leq b(x) \leq 1 \text{ in } \R^n, \\
b(x) &= b(-x) \text{ in } \R^n, \\
b(x) &= 0 \text{ in } \R^n  \setminus B_{1/2},
\end{align*}

As in the proof of Lemma \ref{l:truncatedLisSubsolution},
\[ I \left(u \ast  \frac{\phi_k(x) b(x)}{\norm{\phi_k(x) b(x)}_{L^1}} \right) \geq 0\]
Therefore $\Mp_2 \left( u \ast (\phi_k b) - \norm{\phi_k b}_{L^1} u \right) \geq 0$.

On the other hand, we estimate $I(u \ast (\phi_k(1-b)))$ in $B_{1/2}$.
\begin{align*}
I(u \ast  (\phi_k(1-b))) &= \inf_\ind \int_{\R^n} L_\ind(u \ast (\phi_k(1-b))) \dd y \\
&= \inf_\ind \int_{\R^n} u \ast L_\ind(\phi_k(1-b)) \dd y
\end{align*}
Since by \eqref{e:ellipticity}, $\phi_k(1-b) \in L^1$ and by \eqref{e:kernelc2}, $D^2 \phi_k (1-b) \in L^1$ uniformly in $k$, then $L(\phi_k(1-b))$ is bounded in $L^1$ uniformly for all $L \in \LI_0$. Therefore,
\[ \abs{I(u \ast  (\phi_k(1-b)))} \leq C \norm{u}_{L^\infty} .\]

Using Lemma 5.8 in \cite{CS} and the homogeneity of $\Mm$,
\[ \Mm_2 \left( u \ast (\phi_k (1-b)) - \norm{\phi_k (1-b)}_{L^1} u \right) \geq -C \norm{u}_{L^\infty} \]
Therefore
\begin{align*}
\Mp_2 \left( u \ast \phi_k - \norm{\phi_k}_{L^1} u \right) &\geq \Mp_2 \left( u \ast (\phi_k b) - \norm{\phi_k b}_{L^1} u \right) + \Mm_2 \left( u \ast (\phi_k (1-b)) - \norm{\phi_k (1-b)}_{L^1} u \right) \\
&\geq -C \norm{u}_{L^\infty}
\end{align*}

We finish the proof of the lemma by taking $k \to \infty$ using Lemma 4.3 in \cite{CS2}.
\end{proof}

\begin{lemma} \label{l:eachLaBounded} 
Let $u$ be a solution of \eqref{e:main}. Assume $L_\ind \in \LI_2$ for every $\ind$. Then for every $\ind$, $L_\ind u \leq C \norm{u}_{L^\infty}$ in $B_{1/8}$ for some universal constant $C$. 
\end{lemma}

\begin{proof}
By Lemma \ref{l:approximationByC2}, we can assume the function $u$ is $C^2$. Indeed, for every $\eps > 0$ we can approximate $u$ with a $C^2$ function $u^\eps$ that satisfies the same kind of equation. If we can prove the estimate for $u^\eps$ with a universal constant $C$ that does not depend on $\eps$, then we would prove it for $u$ by passing to the limit as $\eps \to 0$. So we assume that $u \in C^2$ and thus all the integrals are well defined.

From Lemma \ref{l:fullLisSubsolution}, we know that for each $L_\ind$,
\[ \Mp_2(L_\ind u) \geq -C\norm{u}_{L^\infty} \text{ in } B_{1/2}. \]

We would want to apply Theorem \ref{t:weakharnack2} to $L_\ind u$. For that we still need an estimate at least in $L^1((1+|y|)^{-n-\sigma})$. We can easily obtain an estimate in $L^1(B_{1/2})$ using the fact that $L_\ind u \geq 0$ in $B_1$ because of the equation \eqref{e:main}.

Let $b$ be a smooth cutoff function such that $0 \leq b \leq 1$ in $\R^n$, $b=1$ in $B_{1/2}$ and $b=0$ outside $B_1$. We multiply $L_\ind u$ by $b$ and \emph{integrate by parts}.
\begin{align*}\int_{\R^n} L_\ind u(x) \ b(x) \dx \leq \int_{\R^n} L_\ind b(x) \ u(x) \dx \leq C \norm{u}_{L^\infty}
\end{align*} 
for some universal constant $C$. Since $L_\ind u \geq 0$ in $B_1$, then it is in $L^1(B_{1/2})$.

We would still need some control on the values of $L_\ind u$ away from $B_{1/2}$ in order to apply theorem \ref{t:weakharnack2}. We do not want to assume any regularity for $u$ outside $B_1$, so our only choice is to cut off again.

Let $c(x) := b(2x)$ and $w(x) = c(x) \: L_\ind u(x)$ %be a bump function with a smaller domain than $b$.
We will estimate $\Mp_2 w(x)$ for $x \in B_{1/4}$. For that, let us consider any operator $L \in \LI_2$ and estimate
\begin{align*}
L w(x) &= \int_{\R^n} \si w(x,y) K(y) \dd y \\
&= \int_{\R^n} \si L_\ind u(x,y) K(y) \dd y - \int_{\R^n} \si (L_\ind u (1-c))(x,y) K(y) \dd y \\
&\geq -C\norm{u}_{L^\infty} - 2\int_{\R^n} L_\ind u(x+y) (1-c(x+y)) K(y) \dd y \\
&\geq -C\norm{u}_{L^\infty} - 2\int_{\R^n} u(x+y) L_\ind \left( (1-c(x+.)) K \right) \dd y \\
&\geq -C\norm{u}_{L^\infty}
\end{align*}

For the last inequality it was used that $L_\ind \left( (1-c(x+.)) K \right)$ is in $L^1$ uniformly for $x \in B_{1/4}$. This follows from the estimates \eqref{e:ellipticity} and \eqref{e:kernelc2} of the kernel $K$.

Now we can apply Theorem \ref{t:weakharnack2} to $w$ to obtain that $w \leq C \norm{u}_{L^\infty}$ in $B_{1/8}$. but $w=L_\ind u$ in $B_{1/4}$, so we finish the proof.
\end{proof}

\section{Extremal operators are bounded}
\label{s:ExtremalOperatorsAreBounded}

In this section we prove Theorem \ref{t:MabsBounded}. We will achieve this result by showing that $Lu(x)$ is bounded uniformly for all $L \in \LI_0$. The ideas are very similar to section \ref{s:eachLaBounded} but now we apply them to operators that are not a priori bounded below, so we use the $L^2$ estimates from section \ref{s:linearTheory} instead.

\begin{lemma} \label{l:locallizedSubsolution}
Assume $u \in C^2$. Then for every kernel $K$ corresponding to an operator $L \in \LI_0$ and every bump function $b$ such that
\begin{align*}
0 &\leq b(x) \leq 1 \text{ in } \R^n, \\
b(x) &= b(-x) \text{ in } \R^n, \\
b(x) &= 0 \text{ in } \R^n  \setminus B_{1/2},
\end{align*} 
then
\[ \Mp_2 \left(  b(x) \int_{B_{1/2}} \si u(x,y) K(y) \dd y \right) \geq -C \norm{u}_{L^\infty} \qquad \text{in } B_{1/2}\]
\end{lemma}

\begin{proof}
Let us call \[L^t u(x) = \int_{B_{1/2}} \si u(x,y) K(y) \dd y.\] By Lemma \ref{l:truncatedLisSubsolution}, we have
$\Mp_2 (L^t u) \geq 0$. Let $L$ be any operator in $\LI_2$, so we estimate
\begin{align*}
L  (b L^t u) (x) &= \int_{\R^n} \si (L^t u)(x,y) K(y) \dd y - \int_{\R^n} \si ((1-b)L^t u) (x,y) K(y) \dd y \\
&\geq - 2 \int_{\R^n} (1-b(x+y))L^t u(x+y) K(y) \dd y \\
&\geq - 2 \int_{\R^n} u(x+y) L^t((1-b(x+\cdot)) K)(y) \dd y \\
&\geq -C \norm{u}_{L^\infty}
\end{align*}
where we used that $L^t((1-b(x+\cdot)) K)$ is bounded in $L^1$ uniformly in $x$. This is due to the fact that $D^2 ((1-b(x+\cdot)) K) \in L^1(\R^n)$ because of \eqref{e:kernelc2}.
\end{proof}

\begin{lemma} \label{l:everyLisBounded}
Let $u$ be a solution of \eqref{e:main} with all operators $L_\ind$ in $\LI_2$. There is a constant $C$ such that for every operator $L$ in $\LI_0$ we have
\[ |L u(x)| \leq C \norm{u}_{L^\infty} \qquad \text{in } B_{1/2} \]
\end{lemma}

\begin{proof}
As in the proof of Lemma \ref{l:eachLaBounded}, we can and will assume that $u \in C^2$. We will write the proof assuming that $\norm{u}_{L^\infty(\R^n)}=1$. Moreover, we will prove the estimate in $B_{1/64}$ instead of $B_{1/2}$. The general estimate follows directly by scaling and a standard covering argument.

Let $L_\ind$ be one of the operators used in the infimum in \eqref{e:main}. We know from Lemma \ref{l:eachLaBounded} that $L_\ind$ is bounded in $B_{1/2}$ by a constant $C$. In particular $\norm{L_\ind u}_{L^2(B_{1/2})} \leq C$. From Theorem \ref{t:linearL2estimate-local}, we have an $L^2$ estimate for every $L \in \LI_0$,
\[ \norm{L u}_{L^2(B_{1/4})} \leq C. \]

We split the integral of $Lu$ into two domains
\[ Lu(x) = \int_{B_{1/2}} \si u(x,y) K(y) \dd y + \int_{\R^n \setminus B_{1/2}} \si u(x,y) K(y) \dd y. \]

The second integral is clearly bounded since $K(y)$ is a function in $L^1(\R^n \setminus B_{1/2})$. Thus we still have an estimate in $L^2$ for the first term.
\[ \norm{ \int_{B_{1/2}} \si u(x,y) K(y) \dd y }_{L^2(B_{1/4})} \leq C \]

On the other hand, from Lemma \ref{l:truncatedLisSubsolution}, 
\[ \Mp_2 \left( \int_{B_{1/2}} \si u(x,y) K(y) \dd y \right) \geq 0 \]

For a bump function $c(x)$ such that
\begin{align*}
\supp c &= B_{1/4} \\
c &\equiv 1 \ \text{in } B_{1/8} 
\end{align*}
we define
\[ w(x) := c(x) \int_{B_{1/2}} \si u(x,y) K(y) \dd y \]

We know that $w \in L^2(\R^n)$ and $w = 0$ outside $B_{1/4}$. In particular $w$ is bounded in the weighted $L^1$ space: $L^1((1+|y|^{n+\sigma})^{-1})$ needed for Theorem \ref{t:weakharnack2}.

We estimate $\Mp_2 w \geq -C$ in $B_{1/16}$ now in the same way as in the proof of Lemma \ref{l:eachLaBounded} (using assumption \eqref{e:kernelc2}). Thus we can apply theorem \ref{t:weakharnack2} to $w$ to obtain that $w \leq C$ in $B_{1/32}$, which naturally implies $L^k u \leq C$ in $B_{1/32}$. Passing to the limit we get that $Lu \leq C$ for any $L \in \LI_0$.

We got the desired bound from above only. In order to get the corresponding bound from below we must use the equation.

Recall that $L_\ind u$ is bounded by Lemma \ref{l:eachLaBounded}, and the formulas of $L_\ind$ and $L$ are given by
\begin{align*}
L_\ind u(x) = \int \si u(x,y) K_\ind(y) \dd y \\
L u(x) = \int \si u(x,y) K(y) \dd y
\end{align*}
where both $K$ and $K_\ind$ are bounded below by $(2-s)\lambda/|y|^{n+\sigma}$ and above by $(2-s)\Lambda/|y|^{n+\sigma}$.

Consider the kernel \[K_d = \frac 2 \lambda K_\ind - \frac 1 \Lambda K\]
and the corresponding linear operator $L_d$. The kernel $K_d$ satisfies the ellipticity conditions $(2-\sigma)/|y|^{n+\sigma} \leq K_d \leq (2-\sigma)(2\Lambda/\lambda - \lambda/\Lambda)/|y|^{n+\sigma}$, so $L_d$ is in the class $\LI_0$ with ellipticity constants $1$ and $(2\Lambda/\lambda - \lambda/\Lambda)$. The same proof as above tells us that $L_d u \leq C$ in $B_{1/32}$. But then since $L_\ind$ is bounded, then we obtain a bound below for $L$ in $B_{1/32}$.
\[ Lu = 2 \frac \Lambda \lambda L_\ind - \Lambda L_d \geq -C \]

Thus, we have both bounds and we obtain $|\norm{Lu}| \leq C$ in $B_{1/32}$.
\end{proof}

\begin{cor}\label{c:MpAndMmBounded}
$\Mp_0 u$ and $\Mm_0 u$ are bounded in $B_{1/2}$. 
\end{cor}

\begin{proof}
Since $\Mp_0 u = \sup_{L \in \LI_0} Lu$ and for every $L$ in $\LI_0$ we have $|L u| \leq C \norm{u}_{L^\infty}$ with $C$ independent of the choice of $L$ in $\LI_0$, then also $|\Mp_0 u| \leq C \norm{u}_{L^\infty}$ with the same constant $C$.
\end{proof}

Now we are going to prove that all integrals in \eqref{e:main} are absolutely convergent. This already implies that the solution is classical in some way.

\begin{thm} \label{t:MabsBounded}
Assume every $L_\ind$ in \eqref{e:main} belongs to the class $\LI_2$. If $u$ is a bounded function in $\R^n$ such that $Iu=0$ in $B_1$ in the viscosity sense, then we have the following estimate
\[ \int_{\R^n} |\si u(x,y)| \frac{(2-s)}{|y|^{n+\sigma}} \dd y \leq C \norm{u}_{L^\infty(\R^n)} \ \text{ in } B_{1/2}\]
\end{thm}

\begin{proof}
Applying Lemma \ref{l:everyLisBounded} to $L = -(-\lap)^{\sigma/2}$ we get
\[ |-(-\lap)^{\sigma/s} u(x)| = \abs{\int_{\R^n} \si u(x,y) \frac{(2-s)}{|y|^{n+\sigma}} \dd y} \leq C  \norm{u}_{L^\infty(\R^n)} \ \ \text{in } B_{1/2}.\]

On the other hand, applying Corollary \ref{c:MpAndMmBounded} with any pair $\lambda < \Lambda$
\[ |\Mp_0 u(x)| = \abs{\int_{\R^n} (\Lambda \si u(x,y)^+ - \lambda \si u(x,y)^-) \frac{(2-s)}{|y|^{n+\sigma}} \dd y} \leq C \norm{u}_{L^\infty(\R^n)} \ \ \text{in } B_{1/2}. \]

Subtracting, we obtain
\begin{align*}
 \Mp_0 u(x) + \lambda (-\lap)^{\sigma/2} u(x) &\leq C  \norm{u}_{L^\infty(\R^n)}\\
 (\Lambda - \lambda) \int_{\R^n} \si u(x,y)^+ \frac{(2-s)}{|y|^{n+\sigma}} \dd y &\leq C \norm{u}_{L^\infty(\R^n)}
\end{align*} 

On the other hand, by subtracting $\Lambda (-\lap)^{\sigma/2} u(x) - \Mp_0 u(x)$ we obtain the bound

\[ (\Lambda - \lambda) \int_{\R^n} \si u(x,y)^- \frac{(2-s)}{|y|^{n+\sigma}} \dd y = \Lambda (-\lap)^{\sigma/2} u(x) - \Mp_0 u(x) \leq C  \norm{u}_{L^\infty(\R^n)}\]

Combining the two estimates above, we finish the proof.
\end{proof}

\section{Outline of the strategy: the second order case.}
Theorem \ref{t:MabsBounded} provides an estimate slightly stronger than $u \in C^\sigma$. In the case $\sigma \to 2$ it becomes an estimate of the $C^{1,1}$ norm of $u$. Comparing with the proof of Evans-Krylov theorem (as in \cite{MR649348}, \cite{MR661144} or \cite{CC}), the underlying strategy of the proof up to this point is essentially the same but adapted to the integro-differential setting using the ideas in our previous papers \cite{CS} and \cite{CS2}.

The next step in the proof of our main result is to pass from this $C^\sigma$ estimate to a $C^{\sigma+\alpha}$ estimate. In the second order case it corresponds to the apriori estimate $\norm{u}_{C^{2,\alpha}(B_{1/2})} \leq C \norm{u}_{C^{1,1}(B_1)}$. The presently known proofs of this apriori estimate seem difficult to adapt to the nonlocal setting. Thus we present a different strategy for the proof. The key tools that the proof is based on are similar, but in our approach they are organized differently and arguably more directly. We plan to publish a short note \cite{cs-toappear} focusing only on this new proof for concave second order elliptic equations.

In order to better understand our proof in the next section, we first sketch its adaptation to the second order case. 

\newcommand{\ee}{\sigma}

We consider a $C^2$ solution of a fully nonlinear equation $F(D^2 u)=0$ with $F$ concave and uniformly elliptic. We transform this into an integral equation by pointing out that a linear equation can be written as an integral on the unit sphere $S_1$.
\[ a_{ij} \partial_{ij} u(x) = \int_{S_1} \partial_{\ee \ee} u(x) \ w(\ee) \dd \ee \]
with the weight $w(\ee) = 1/(\det \{a_{ij}\} a^{ij} \ee_i \ee_j)$ where $\{a^{ij}\} = \{a_{ij}\}^{-1}$. If the coefficients $a_{ij}$ are uniformly elliptic, then $w(\ee)$ will be bounded away from zero.

We recall that $F$ being concave implies that pure second derivatives are all subsolutions of the linearized operator. In particular, for any fixed set $A \subset S_1$, so is
\[ v_A = \int_A \partial_{\ee \ee} u(x) \dd \ee. \]

We also recall that for a (non necessarily concave) fully nonlinear equation
\[F(D^2 u) = \sup_b \inf_\ind L_{\ind b} u = 0,\] a solution $u$ satisfies (just because it is an $\inf \sup$) that for any two points $x$ and $y$ in the domain, there exists an operator $L_{\ind b}$ for which
\[L_{\ind b} u(x) - L_{\ind b} u(y) \geq 0. \]
In our approach, this means that there is a weight $w(\ee)$, bounded below and above depending on the ellipticity constants, such that
\[ \int_{S_1} \left( \partial_{\ee \ee} u(x) - \partial_{\ee \ee} u(y) \right) w(\ee) \dd \ee \geq 0 . \]
In particular, since there is another weight which gives the same inequality exchanging $x$ and $y$, we must have that the following quantities are comparable:
\begin{equation}  \label{e:2ndOrderComparable}
 \int_{S_1} \left( \partial_{\ee \ee} u(x) - \partial_{\ee \ee} u(y) \right)^+ \dd \ee \approx \int_{S_1} \left( \partial_{\ee \ee} u(x) - \partial_{\ee \ee} u(y) \right)^- \dd \ee 
\approx \int_{S_1} \abs{ \partial_{\ee \ee} u(x) - \partial_{\ee \ee} u(y) } \dd \ee
\end{equation}

At this point we define
\begin{align*}
h(x,\ee) &= \partial_{\ee \ee} u(x) - \partial_{\ee \ee} u(0) \\
w_A(x) &= \int_A h(x,\ee) \dd \ee 
\end{align*}
for any set $A \subset S_1$. We will use only the properties above to show that 
\[ \int_{S_1} |h(x,\ee)| \dd \ee \leq |x|^\alpha. \]
The $C^{2,\alpha}$ estimate for $u$ follows easily from this estimate.

By \eqref{e:2ndOrderComparable}, we only need to prove that $w_A(x) \leq C |x|^\alpha$ for every set $A$, since
\[ \int_{S_1} |h(x,\ee)| \dd \approx  \int_{S_1} h(x,\ee)^+ \dd \ee = \sup_A w_A(x) \]

In fact, by renormalization we only need to prove the following lemma.

\begin{lemma}
Assume that for $x$ in $B_1$, for any set $A \subset S_1$, $w_A(x) \leq 1$. Then there is a universal constant $\theta>0$ such that
\[ w_A(x) \leq 1-\theta\]
for any $A \subset S_1$ and $x \in B_{1/2}$.
\end{lemma}

\begin{proof}[Sketch of the proof]
Suppose that there exists an $x \in B_{1/2}$ where $w_A > 1-\theta$ for some set $A \subset S_1$. We will arrive to a contradiction if $\theta$ is too small.

Since $w_A \leq 1$ in $B_1$ and $w_A$ is a subsolution of the linearized equation, we can apply Theorem 4.8(1) in \cite{CC} (the $L^\eps$ estimate) to $1-w_A$ (this will correspond to Theorem 10.4 in \cite{CS} in the nonlocal case). It follows that we can make
\[ \Omega = \{ w_A(x) \geq 1 - t \theta\} \]
cover almost all $B_{1/4}$ if we choose $t$ large (but independently of $\theta$).

We will now obtain a contradiction by looking at $w_{A^c}$ in $B_{1/4}$.

For every $x$ in $\Omega$, the choice of the set $A$ is almost maximal in the sense that
\[1-t \theta \leq w_A(x) \leq \int_{S_1} h(x,\ee)^+ \dd \ee \leq 1 \]

On the other hand, since
\[\int_{S_1} h(x,\ee) \dd \ee = \int_{S_1} h(x,\ee)^+ \dd \ee - \int_{S_1} h(x,\ee)^- \dd \ee = w_A(x) + w_{A^c}(x) \]
then also
\[ 0 \leq w_{A^c} + \int_{S_1} h(x,\ee)^- \dd \ee \leq t \theta \]

From \eqref{e:2ndOrderComparable}, we know that the integrals of $h^+$ and $h^-$ are comparable. Thus in $\Omega$ we have
\[ w_{A^c} \leq t \theta - C \]
for a constant $C$ depending on $\lambda$ and $\Lambda$. If we choose $\theta$ small, that means that $w_{A^c}$ will be strictly negative in most of $B_{1/4}$. But then applying Theorem 4.8(2) in \cite{CC} (which corresponds to Theorem \ref{t:weakharnack2} in the nonlocal case) we obtain that $w_{A^c}(0) \leq -c$ for some universal constant $c$. This is a contradiction since clearly $w_{A^c}(0)=0$.
\end{proof}

In the integro-differential case, the proof will be slightly lengthier in part because we have to keep track of the truncation error we make every time we localize an integral. We cover the proof in detail in the next section.

\section{Further regularity.}
This section is devoted to fill the gap between Theorem \ref{t:MabsBounded} and Theorem \ref{t:main}.

From Theorem \ref{t:MabsBounded}, we know that
\begin{equation} \label{e:f1}
 \int_{B_{1/2}} |\si u(x,y)| \frac{2-\sigma}{|y|^{n+\sigma}} \dd y \leq C \ \text{ in } B_{1/4}
\end{equation}

Our objective is to show that
\[ \int_{B_{1/2}} |\si u(x,y) - \si u(0,y)| \frac{2-\sigma}{|y|^{n+\sigma}} \dd y \leq C |x|^\alpha \]
for some constant $C$ and $\alpha>0$ and for every $x \in B_{1/4}$. This estimate implies the H\"older continuity of the fractional laplacian $(-\lap)^{\sigma/2}$ from which the $C^{\sigma+\alpha}$ regularity of $u$ follows.

We will consider all kernels $K$ of the form
\[ K_A(y) = \frac{(2-\sigma)}{|y|^{n+\sigma}} \chi_A(y) \]
where $\chi_A(y)$ is the characteristic function of an arbitrary set $A$.

Let $b$ be a bump function as in Lemma \ref{l:locallizedSubsolution}. For each set $A$, we write
\[ w_A(x) = b(x) \int_{B_{1/2}} (\si u(x,y) - \si u(0,y)) K_A(y) \dd y \]

We know that $w_A$ is uniformly bounded from Lemma \ref{l:everyLisBounded}. From Lemma \ref{l:locallizedSubsolution}, we have
\begin{equation}
\Mp_2 w_A \geq -C \ \text{ in } B_{1/4} \text{ uniformly in } A \label{e:eqForW}
\end{equation}

We define the following quantities
\begin{align*}
P(x) &:= \sup_A w_A(x) = b(x) \int_{B_{1/2}} (\si u(x,y) - \si u(0,y))^+ \frac{(2-\sigma)}{|y|^{n+\sigma}} \dd y \\ 
N(x) &:= \sup_A -w_A(x) = b(x) \int_{B_{1/2}} (\si u(x,y) - \si u(0,y))^- \frac{(2-\sigma)}{|y|^{n+\sigma}} \dd y
\end{align*}

Note that $P(x)$ is realized by the set $A = \{ x : \si u(x,y) > \si u(0,y) \}$ and $N(x)$ is realized by the complement of that set.

\begin{lemma} \label{l:PandN}
There is a constant $C$ such that for $x \in B_{1/4}$,
\[ \frac \lambda \Lambda N(x) - C|x| \leq P(x) \leq \frac \Lambda \lambda N(x) + C|x|. \]
\end{lemma}

\begin{proof}
For some $x \in B_{1/4}$, let $u_x(z) := u(x+z)$. Since $u$ solves the equation \eqref{e:main} in a neighborhood of $x$, then both $u$ and $u_x$ solve \eqref{e:main} in a neighborhood of $0$. Thus $\Mp_2 (u_x - u)(0) \geq 0$ and $\Mm_2 (u_x - u)(0) \leq 0$.

For every kernel $K$ in the family $\LI_2$ we have
\begin{align*}
L(u_x - u)(0) &= \int_{\R^n} (\si u(x,y) - \si u(0,y)) K(y) \dd y \\
&= \int_{B_{1/2}} (\si u(x,y) - \si u(0,y)) K(y) \dd y \\
&\qquad + \int_{\R^n \setminus B_{1/2}} (\si u(x,y) - \si u(0,y)) K(y) \dd y
\end{align*}

Let us analyze the second term in the right hand side.
\begin{align*}
\int_{B_{1/2}^c} &(\si u(x,y) - \si u(0,y)) K(y) \dd y = \int_{\R^n} \si u(0,y) \left(K(y-x) \chi_{B_{1/2}^c} (y-x) +K(y) \chi_{B_{1/2}^c} (y-x)  \right) \dd y \\
&\leq  \int_{\R^n \setminus B_{1/2 + |x|}} |\si u(0,y)| \frac{C}{|y|^{n+\sigma+1}} |x| \dd y 
+ 8 \norm{u}_{L^\infty} \int_{B_{1/2 + |x|} \setminus B_{1/2}}  \frac{\Lambda (2-\sigma)}{|y|^{n+\sigma}} \dd y 
\\&\leq  C |x|
\end{align*}
Therefore, for every kernel $K$ in the family $\LI_2$, we have
\[ \int_{\R^n} (\si u(x,y) - \si u(0,y)) K(y) \dd y \geq \int_{B_{1/2}} (\si u(x,y) - \si u(0,y)) K(y) \dd y + C|x| \]

Taking the supremum we obtain
\[ 0 \leq \Mp_2(u_x-u) \leq \sup_K \int_{B_{1/2}} (\si u(x,y) - \si u(0,y)) K(y) \dd y + C|x| \]

In particular, if we take the suppremum over all kernels $K$ in $\LI_0$ (a larger family), we still have
\[
\sup_{\lambda \frac{(2-\sigma)}{|y|^{n+\sigma}} \leq K \leq \Lambda \frac{(2-\sigma)}{|y|^{n+\sigma}}} \int_{B_{1/2}} (\si u(x,y) - \si u(0,y)) K(y) \dd y  \geq -C|x|
\]
which is the same as $\Lambda P(x) - \lambda N(x)  \geq -C|x|$.

The same computation with $\Mm_2(u_x-u)(0) \leq 0$ provides the other inequality.
\end{proof}

It is important to notice the following relation,
\begin{align*}
 \int_{B_{1/2}} |\si u(x,y) - \si u(0,y)| \frac{2-\sigma}{|y|^{n+\sigma}} \dd y &= \sup_A w_A - \inf_A w_A \\
 &= P(x) + N(x)
\end{align*}

The strategy for proving our regularity result will be to prove that $\sup_{x \in B_r} P(x) \leq Cr^\alpha$. It is enough to prove it for $|x|$ small enough, therefore we can consider a rescaled situation by taking $\bar w_A(x) = \frac 1 C w_A(rx)$ where $C$ is the constant from \eqref{e:f1} and $r$ is small enough so that our estimates become
\begin{align}
&\text{for every $K$: } |w_K| \leq 1 \text{ in } \R^n \label{e:wKbounded} \\ 
&\text{for every $K$: } \Mp_2 w_K \geq -\eps_1 \text{ in } B_1 \label{e:wKsubsolution}  \\
& \frac \lambda \Lambda N(x) - \eps_1 |x|^{1-\eps_1} \leq P(x) \leq \frac \Lambda \lambda N(x) + \eps_1 |x|^{1-\eps_1} \label{e:PandNcontrolled}
\end{align}
for $\eps_1$ arbitrarily small.

\begin{lemma} \label{ultimateLemma}
Assume $\sigma \in (1,2)$. Let $P(x)$ be the function defined above. There is a constant $C$ and $\alpha>0$ such that
\[ P(x) \leq C r^\alpha \]
\end{lemma}

\begin{proof}
As mentioned above, after an appropriate scaling, we can assume that \eqref{e:wKbounded}, \eqref{e:wKsubsolution} and \eqref{e:PandNcontrolled} hold with $\eps_1$ arbitrarily small. On the other hand, given the construction in Lemma \ref{l:approximationByC2}, we can assume $u$ is $C^2$ and thus $w_K$, $P$ and $N$ are continuous. We will obtain the apriori estimates independently of the modulus of continuity of them, so the estimate holds when passing to the limit.

We will prove that there is $r>0$ and $\theta>0$ such that
\begin{equation} \label{a5}
 \sup_{B_{r^k}} |P| \leq (1-\theta)^k = r^{\alpha k} \text{ where } \alpha = \frac{\log (1-\theta)}{\log r}
\end{equation}

This is clear for $k=0$. Let us prove it is true for all values of $k$ by induction. So let us assume it is true up to some value $k$.

Since \eqref{a5} holds up to some value $k$, we have that,
\[ |w_A(x)| \leq (1-\theta)^{-1} |x|^\alpha \text{ for } |x| > r^k \]

Consider the following rescaled functions
\begin{align*}
\tilde w_A(x) &=  (1-\theta)^{-k} w_A(r^k x)\\
\tilde P(x) &=  (1-\theta)^{-k} P(r^k x) = \sup_{A} \tilde w_A (x) \\
\tilde N(x) &=  (1-\theta)^{-k} N(r^k x) = \sup_{A} -\tilde w_A (x)
\end{align*}

The function $\tilde P$ satisfies the relations
\begin{align*}
\tilde P(x) &\leq 1 \qquad \text{ in } B_1 \\
\tilde P(x) &\leq (1-\theta)^{-1} |x|^\alpha \qquad \text{outside } B_1
\end{align*}

Moreover, from \eqref{e:PandNcontrolled},
\begin{equation}\label{e:pyn}
\lambda \tilde N(x) - \eps_1 \leq \tilde P(x) \leq \Lambda \tilde N(x) +\eps_1
\end{equation}

We want to show that if $\theta$ and $r$ are chosen small enough we will have $\tilde P \leq (1-\theta)$ in $B_r$. The proof is by contradiction. We will arrive to a contradiction if $\theta$ and $r$ are small enough.

Let $x_0$ be the point where the maximum of $\tilde P$ is achieved in $\overline B_r$ for some $r \in (0,1/2)$. We assume $\tilde P(x_0) \geq 1-\theta$ to get a contradiction. Let $A$ be the set such that $\tilde P(x_0) = \tilde w_A(x_0) \geq 1-\theta$.

Let $v_A = (1-\tilde w_A)^+$. We know that $\inf_{B_r} v_A \geq \theta$. Moreover,
\begin{align*}
\Mm_2 v_A &\leq \Mm_2 (1-\tilde w_A) - \Mm_2 (1-\tilde w_A)^- \\
&\leq -\Mp_2 \tilde w_A - \Mm_2 (1-\tilde w_A)^-\\
&\leq C \ \text{ in } B_{1/2}
\end{align*}
since $\Mp \tilde w_A \geq -\eps_1$ and $(1-\tilde w_A)^- \leq ((1-\theta)^{-1} |x|^\alpha - 1)^+$.

By Theorem 10.4 in \cite{CS}, for some $p>0$ and $r < 1/4$ we have the estimate,
\[ |\{ v_A > t \theta \} \cap B_{2r} | \leq C r^n (\theta + C r^\sigma)^p (t\theta)^{-p} \]
Let us choose $r$ (depending on $\theta$ to be chosen later) so that $C r^\sigma < \theta$. So that we have
\[ |\{ v_A > t \theta \} \cap B_r | \leq C r^n t^{-p} = c t^{-p} |B_r|. \]
Thus, by choosing $t$ large, we will be able to make the measure of the set $\{ v_A > t \theta \} \cap B_r$ a small factor of $|B_1|$ independently of $\theta$. Note that $v_A > t \theta$ is equivalent to $w_A < 1 - t \theta$.

Let $G = \{ v_A \leq t \theta \} \cap B_r$. We know that $|G| \geq (1-c t^{-p})|B_r|$. The set $G$ is also the set where $\tilde w_A \geq 1-t \theta$. On the other hand, since $G \subset B_1$, $\tilde P \leq 1$ in $G$, then $\tilde P - \tilde w_A \leq t \theta$ in $G$. This allows us to estimate the difference between $-N(x)$ and $w_{A^c}$ in $G$, where $A^c$ is the complement of the set $A$.

Clearly $w_A + w_{A^c} = P - N$, then $N+w_{A^c} = P-w_A \leq t \theta$ in $G$. Since $\tilde N(x) \geq \lambda / \Lambda \tilde P(x) - \eps_1$, we have that in $G$
\begin{align*}
\tilde w_{A^c}(x) &\leq -\tilde N(x) + t \theta \\
&\leq -\frac{\lambda}{\Lambda} (1-t \theta) + t \theta + \eps_1 \\
&\leq -\frac{\lambda}{2\Lambda} \qquad \text{if $\theta$ and $\eps_1$ are small enough (depending on $t$)}.
\end{align*}
Consequently, $|\{\tilde w_{A^c} \leq -\frac \lambda {2 \Lambda}\} \cap B_r| \geq (1-ct^{-p}) |B_r|$.

For some small $\kappa>0$, we define $v_c = (\tilde w_A(\kappa r x) + \frac{\lambda}{2\Lambda})^+$. We know $\Mp v_c \geq -\eps_1$ in $B_2$, thus we can apply Theorem \ref{t:weakharnack2} to $v_c(\kappa r x)$ for some small $R>0$ and get
\begin{align*}
v_c(0) &\leq C \eps_1 + C \int_{\R^n} \frac{|v_c(x)|}{1+|y|^{n+\sigma}} \dd y \\
&\leq C \eps_1 + C \int_{|y| \leq \kappa^{-1}} \frac{|v_c(x)|}{1+|y|^{n+\sigma}} \dd y + C \int_{|y|>\kappa^{-1}} \frac{|v_c(x)|}{1+|y|^{n+\sigma}} \dd y \\
\intertext{using that $| \{v_c>0\} \cap B_{\kappa^{-1}} | < C t^{-p} \kappa^{-n}$,}
&\leq C \eps_1 + C \kappa^{-n} t^{-p} + C \int_{|y|>\kappa^{-1}} \frac{\tilde w_A(\kappa r x)^+}{1+|y|^{n+\sigma}} \dd y \\
\intertext{since $r<1$, we can bound the third term independently of $r$,}
&\leq C \eps_1 + C \kappa^{-n} t^{-p} + C \int_{|y|>\kappa^{-1}} \frac{2 (\kappa |y|)^\alpha }{1+|y|^{n+\sigma}} \dd y \\
&\leq C \eps_1 + C \kappa^{-n} t^{-p} + C \kappa^\sigma \\
\end{align*}
So we can choose $\kappa$ and $\eps_1$ so that  $C \eps_1 + C \kappa^\sigma < \lambda / (8\Lambda)$ and then $t$ such that $C \kappa^{-n} t^{-p} < \lambda / (8\Lambda)$. Therefore we got the following estimate
\[ v_c(0) \leq \frac{\lambda}{4\Lambda}. \]
But this means that $\tilde w_K(0) \leq - \frac{\lambda}{4\Lambda}$ which is a contradiction since $\tilde w_K(0)=0$.

The contradiction came from saying that $\tilde P(x_0) \geq (1-\theta)$ for some $x_0$ in $B_r$. Thus $\tilde P < (1-\theta)$ in $B_r$. In the original scale, this means that $P \leq (1-\theta)^{k+1}$ in $B_{r^{k+1}}$, which finishes the inductive step and the proof.
\end{proof}

Using Lemma \ref{ultimateLemma}, we can finally prove Theorem \ref{t:main}.

\begin{proof}[{\bf Proof of Theorem \ref{t:main}}] 
As it was mentioned before, the case $\sigma \leq 1$ is already covered in \cite{CS}, so we prove the case $\sigma \in (1,2)$ only.

Let us consider the fractional laplacian of order $\sigma$
\[ -(-\lap)^{\sigma/2} u (x) = c_\sigma (2-\sigma) \int_{\R^n} \si u(x,y) \frac{1}{|y|^{n+\sigma}} \dd y \]
where the constant $c_\sigma$ remains bounded below and above for $\sigma \in (1,2)$.

Let $b$ be a bump function as in Lemma \ref{l:locallizedSubsolution}. We rewrite the fractional laplacian as
\begin{align*}
 -(-\lap)^{\sigma/2} u (x) &= \int_{\R^n} \si u(x,y) b(y) \frac{c_\sigma (2-\sigma)}{|y|^{n+\sigma}} \dd y + \int_{\R^n} \si u(x,y) (1-b(y)) \frac{c_\sigma (2-\sigma)}{|y|^{n+\sigma}} \dd y \\
 &= w(x) + g(x)
\end{align*}

%Let $K$ be the following kernel
%\[ K(y) = b(y) \frac{c_\sigma (2-\sigma)}{|y|^{n+\sigma}}. \]
%So $w(x) - w(0) = w_K(x)$. 
By Lemma \ref{ultimateLemma}, $|w(x)-w(0)|\leq P(x) + N(x) \leq C \norm{u}_{L^\infty} |x|^\alpha$ for some $\alpha>0$ and $C$ universal constants.

Since $u \in L^\infty$, then $g \in C^\infty$. In particular $g$ is $C^\alpha$. For a universal constant $C$,
\[ |g(x)-g(0)| \leq C \norm{u}_{L^\infty(\R^n)} |x|^\alpha \]

Adding the two estimates above we get that \[|(-\lap)^{\sigma/2} u(x) - (-\lap)^{\sigma/2} u(0)| \leq C \norm{u}_{L^\infty(\R^n)} |x|^\alpha.\]

Therefore by a standard translation of the estimate we obtained that $(-\lap)^{\sigma/2} u \in C^\alpha (B_{1/2})$, with the estimate
\[\norm{(-\lap)^{\sigma/2} u}_{C^\alpha(B_{1/2})} \leq C \norm{u}_{L^\infty(\R^n)}. \]

But if $(-\lap)^{\sigma/2} u \in C^\alpha$. It is a classical result that this implies a corresponding estimate for $u \in C^{\sigma+\alpha}$ (see for example \cite{stein1970sia}). So we finish the proof.
\end{proof}

\begin{remark}
We have not used the homogeneity of $I$ in any proof in this paper. With the same proof we can obtain the same regularity result for equations of the form
\[ \inf_{\ind \in \mathcal{A}} L_\ind u(x) = \inf_{\ind \in \mathcal{A}} \int_{\R^n} (u(x+y) + u(x-y) - 2u(x)) K_\ind (y) \dd y = 0 + b_\ind\]
for a bounded family of real numbers $b_\ind$. This is the general form of a concave uniformly elliptic nonlocal operator of order $\sigma$.

For the estimates, we would have to include the values of $b_\ind$ in the right hand side:
\[ ||u||_{C^{\sigma+\alpha}(B_{1/2})} \leq C (\norm{u}_{L^\infty} + \sup_{\ind} b_\ind) \]
\end{remark}

\begin{remark}
The assumption $u \in L^\infty(\R^n)$ is not sharp. It could easily be replaced in all estimates by $u \in L^1(\R^n,1/(1+|y|^\sigma))$. We kept the $L^\infty$ norm for simplicity of the exposition. 
\end{remark}

\section{Acknowledgments}

Both authors were partially supported by NSF grants. 

Part of this work took place while both authors were visiting the Institute for Advanced Study during the special program in geometric PDE.

\bibliographystyle{plain}   % Here the bibliography 
\bibliography{nonlek}             % is inserted.
\index{Bibliography@\emph{Bibliography}}%
\end{document}